\documentclass[a4paper]{amsart}

\title[An alternative approach on affine term structure models]{An alternative approach on the existence of affine realizations for HJM term structure models}
\author{Stefan Tappe}
\address{ETH Z\"urich, Department of Mathematics, R\"amistrasse 101, CH-8092 Z\"urich, Switzerland}
\email{stefan.tappe@math.ethz.ch}

\usepackage{amscd}
\usepackage{amsmath}
\usepackage{amssymb}
\usepackage{amsthm}
\usepackage{bbm}
\usepackage{stmaryrd}

\newif\ifpdf
\ifx\pdfoutput\undefined
   \pdffalse        
\else
   \pdfoutput=1     
   \pdftrue
\fi

\ifpdf
   \usepackage[pdftex]{graphicx}
\else
   \usepackage{graphicx}
\fi

\numberwithin{equation}{section} \swapnumbers

\newtheorem{satz}{Satz}[section]

\newtheorem{theorem}[satz]{Theorem}
\newtheorem{proposition}[satz]{Proposition}
\newtheorem{corollary}[satz]{Corollary}
\newtheorem{lemma}[satz]{Lemma}
\newtheorem{assumption}[satz]{Assumption}

\newtheorem{definition}[satz]{Definition}

\newtheorem{remark}[satz]{Remark}

\begin{document}

\maketitle\thispagestyle{empty}

\begin{abstract}
We propose an alternative approach on the existence of affine
realizations for HJM interest rate models. It is applicable to
a wide class of models, and simultaneously it is conceptually rather comprehensible.
We also supplement some known existence results for particular volatility
structures and provide further insights into the
geometry of term structure models.
\bigskip

\end{abstract}


\section{Introduction}

A zero coupon bond with maturity $T$ is a financial asset which pays the holder one unit of cash at $T$. Its price at $t \leq T$ can be written as the continuous discounting of one unit of cash
\begin{align*}
P(t,T) = \exp \bigg( -\int_t^T f(t,s) ds \bigg),
\end{align*}
where $f(t,T)$ is the rate prevailing at time $t$ for instantaneous borrowing at time $T$, also called the forward rate for date $T$. The classical continuous framework for the evolution of the forward rates goes back to Heath, Jarrow and Morton (HJM) \cite{HJM}. They assume that, for every date $T$, the forward rates $f(t,T)$ follow an It\^o process of the form
\begin{align*}
df(t,T) = \alpha_{\rm HJM}(t,T) dt + \sigma(t,T) dW_t,\quad t\in [0,T]
\end{align*}
where $W$ is a Wiener process. Note that such an HJM interest rate model is an infinite dimensional object, because for every date of maturity $T \geq 0$ we have an It\^{o} process.

There are several reasons why, in practice, we are interested in the existence of a finite dimensional realization, that is, the forward rate evolution can be described by a finite dimensional state process. Such a finite dimensional realization ensures larger analytical tractability of the model, for example, in view of option pricing. Moreover, as argued in \cite{bautei:05}, HJM models without a finite dimensional realization do not seem reasonable, because then the support of the forward rate curves $f(t,t+\cdot)$, $t > 0$ becomes to large, and hence any ``shape'' of forward rate curves, which we assume from the beginning to model the
market phenomena, is destroyed with positive probability.

The problem concerning the existence and construction of finite
dimensional realizations for HJM interest
rate models has been studied, for various special cases,
in \cite{Jeffrey, Ritchken, Duffie-Kan, Bhar, Inui, Bj_Ch, Bj_Go,
Kwon_2001, Kwon_2003}, and has finally completely been solved in
\cite{Bj_Sv, Bj_La, Filipovic}, see also \cite{Bjoerk} for a survey.

The main idea is to switch to the Musiela parametrization of forward curves
$r_t(x) = f(t,t+x)$ (see \cite{Musiela}), and to consider the forward rates as the solution of
a stochastic partial differential equation (SPDE), the so-called
HJMM (Heath--Jarrow--Morton--Musiela) equation
\begin{align}\label{HJMM}
\left\{
\begin{array}{rcl}
dr_t & = & (\frac{d}{dx} r_t + \alpha_{\rm HJM}(r_t))dt +
\sigma(r_{t})dW_t
\medskip
\\ r_0 & = & h_0
\end{array}
\right.
\end{align}
on a suitable Hilbert space $H$ of forward curves, where
$\frac{d}{dx}$ denotes the differential operator, which is generated
by the strongly continuous semigroup $(S_t)_{t \geq 0}$ of shifts.

The bank account $B$ is the riskless asset, which starts with one unit of cash and grows continuously at time $t$ with the short rate $r_t(0)$, i.e.
\begin{align*}
B(t) = \exp \bigg( \int_0^t r_s(0) ds \bigg), \quad t \geq 0.
\end{align*}
According to \cite{ds94}, the implied bond market, which we can now express as
\begin{align}\label{bond-market}
P(t,T) = \exp \bigg( -\int_t^{T-t} r_t(x)dx \bigg), \quad 0 \leq t
\leq T
\end{align}
is free of arbitrage if there exists an equivalent (local)
martingale measure $\mathbb{Q} \sim \mathbb{P}$ such that the discounted bond prices
$$
\frac{P(t,T)}{B(t)}, \quad t \in [0,T]
$$
are local $\mathbb{Q}$-martingales for all maturities $T$. If we formulate the
HJMM equation (\ref{HJMM}) with respect to such an equivalent
martingale measure $\mathbb{Q} \sim \mathbb{P}$, then the drift is determined by the
volatility, i.e. $\alpha_{\rm HJM} : H \rightarrow H$ in
(\ref{HJMM}) is given by the HJM drift condition (see \cite{HJM})
\begin{align}\label{HJM-drift}
\alpha_{\rm HJM}(h) &= \sigma(h) \int_0^{\bullet} \sigma(h)(\eta)
d\eta = \frac{1}{2} \frac{d}{dx} \bigg( \int_0^{\bullet}
\sigma(h)(\eta) d\eta \bigg)^2, \quad h \in H.
\end{align}
Now, we can consider the problem from a geometric point of view, and
the existence of a finite dimensional realization just means the
existence of an invariant manifold, i.e. a finite dimensional submanifold, which the forward rate process never leaves. Applying the Frobenius Theorem
we obtain the following necessary and sufficient condition for the
existence of an invariant manifold, namely
\begin{align*}
{\rm dim} \{ \beta,\sigma \}_{\rm LA} < \infty,
\end{align*}
i.e. the so-called Lie algebra generated by the vector fields
\begin{align*}
h \mapsto \beta(h) := \frac{d}{dx} h + \alpha_{\rm HJM}(h) - \frac{1}{2}
D\sigma(h)\sigma(h)
\end{align*}
and $h \mapsto \sigma(h)$ must be locally of finite dimension. These
are the essential ideas of the mentioned articles \cite{Bj_Sv,
Bj_La, Filipovic}.

The technical problem with this approach is that the differential
operator $\frac{d}{dx}$ is, in general, an unbounded and therefore
non-smooth operator. Bj\"ork et al. \cite{Bj_Sv, Bj_La} choose the
state space $H$ so small such that $\frac{d}{dx}$ becomes bounded.
As a consequence, not all forward curves of basic HJM models belong
to this space, as for example the forward curves implied by a
Cox-Ingersoll-Ross model, see \cite[Sec. 1]{Filipovic}.

Filipovi\'c and Teichmann \cite{Filipovic} solved this problem by
using convenient analysis on Fr\'echet spaces, developed in
\cite{KriMic:97}, which, however, is far from being trivial to carry
out. In their paper, they in particular show that any HJM model with
a finite dimensional realization necessarily has an affine term
structure.

The contribution of the present paper is to propose an alternative
approach, which is characterized by the following two major
features:

\begin{itemize}
\item We work on the Hilbert space $H$ from \cite[Sec. 5]{fillnm}, which
is large enough to capture any reasonable forward curve. As already mentioned, Bj\"ork et al. \cite{Bj_Sv, Bj_La} choose the space
$H$ such that the differential operator $\frac{d}{dx}$ is bounded,
whence it is rather small.

\item Simultaneously, this article does not require knowledge about
convenient analysis on Fr\'echet spaces. This rather technical machinery is
used in Filipovi\'c and Teichmann \cite{Filipovic}. We avoid this framework by
directly focusing on affine realizations, which, due to the
mentioned result from \cite{Filipovic}, does not mean a restriction.
This makes our approach rather comprehensible.
\end{itemize}

Summing up, we present an alternative approach on the existence of affine realizations
for HJM models, which is applicable to a wide class of models, and which is conceptually accessible to a wide readership.

Our approach also allows us to supplement some existence results for
particular volatility structures from \cite{Bj_Sv} (see our comments
in Remarks \ref{remark-supp-1}, \ref{remark-supp-2}) and to provide
further insights into the geometry of term structure models (see our
comments in Remarks \ref{remark-geom-0}, \ref{remark-geom-1},
\ref{remark-geom-2}).


Before we finish this introduction with overviewing the rest of the paper, let us briefly mention another geometric approach
for modelling zero coupon bonds, which is conceptually completely different from the present HJM framework, but also leads to an invariance problem. It was suggested by Brody and Hughston \cite{brohug} and is inspired by methods from information geometry.
They define the bond prices as
\begin{align*}
P(t,T) = \int_{T-t}^{\infty} \rho_t(x) dx,
\end{align*}
where every $\rho_t$ is a density on $\mathbb{R}_+$. In order to introduce the densities, the authors in \cite{brohug} consider a process $\xi$ on the state space $H = L^2(\mathbb{R}_+)$, which -- by construction -- stays in the positive orthant of the unit sphere $\mathcal{S}$ in the Hilbert space $H$. This implies that $\rho_t = \sqrt{\xi_t}$ is indeed a density.

The rest of this paper is organized as follows. In Section
\ref{sec-foliations} we provide results on invariant foliations and
in Section \ref{sec-real-general} on affine realizations for general
stochastic partial differential equations. In Section
\ref{sec-space} we introduce the space of forward curves. Working on
this space, we present a result concerning invariant foliations for
the HJMM equation (\ref{HJMM}) in Section \ref{sec-HJMM-fol}. Using
this result, we study the existence of affine realizations for the
HJMM equation (\ref{HJMM}) with general volatility in Section \ref{sec-vol-general}, and for various particular
volatility structures in Sections \ref{sec-cdv}--\ref{sec-short-rate}. Finally, Section \ref{sec-conclusion} concludes.

\section{Invariant foliations for general stochastic partial differential
equations}\label{sec-foliations}

In this section, we provide results on invariant foliations for
general stochastic partial differential equations, which we will apply
to the HJMM equation (\ref{HJMM}) later on.

From now on, let $(\Omega,\mathcal{F},(\mathcal{F}_t)_{t \geq
0},\mathbb{P})$ be a filtered probability space satisfying the usual
conditions and let $W$ be a real-valued Wiener process.

Here, we shall deal with stochastic partial
differential equations of the type
\begin{align}\label{SPDE-manifold}
\left\{
\begin{array}{rcl}
dr_t & = & (A r_t + \alpha(r_t))dt + \sigma(r_{t})dW_t
\medskip
\\ r_0 & = & h_0
\end{array}
\right.
\end{align}
on a separable Hilbert space $(H,\| \cdot \|,\langle \cdot,\cdot
\rangle)$. In (\ref{SPDE-manifold}), the operator $A :
\mathcal{D}(A) \subset H \rightarrow H$ is the infinitesimal
generator of a $C_0$-semigroup $(S_t)_{t \geq 0}$ on $H$ with
adjoint operator $A^* : \mathcal{D}(A^*) \subset H \rightarrow H$.
Recall that the domains $\mathcal{D}(A)$ and $\mathcal{D}(A^*)$ are
dense in $H$, see, e.g., \cite[Thm. 13.35.c, Thm. 13.12]{Rudin}.

Concerning the vector fields $\alpha,\sigma : H \rightarrow H$ we
impose the following conditions.

\begin{assumption}\label{ass-SPDE}
We assume that $\alpha,\sigma \in C^1(H)$ and that there is a
constant $L > 0$ such that
\begin{align}\label{Lipschitz-alpha-SPDE}
\| \alpha(h_1) - \alpha(h_2) \| &\leq L \| h_1 - h_2 \|,
\\ \label{Lipschitz-sigma-SPDE} \| \sigma(h_1) - \sigma(h_2) \| &\leq L \| h_1 - h_2 \|
\end{align}
for all $h_1,h_2 \in H$.
\end{assumption}

The Lipschitz assumptions (\ref{Lipschitz-alpha-SPDE}),
(\ref{Lipschitz-sigma-SPDE}) ensure that for each $h_0 \in H$ there
exists a unique weak solution for (\ref{SPDE-manifold}) with $r_0 =
h_0$, see \cite[Thm. 6.5, Thm. 7.4]{Da_Prato}.

\begin{definition}
A subset $U \subset H$ is called {\rm invariant} for
(\ref{SPDE-manifold}) if for every $h \in U$ we have
\begin{align*}
\mathbb{P}(r_t \in U) = 1 \quad \text{for all $t \geq 0$}
\end{align*}
where $(r_t)_{t \geq 0}$ denotes the weak solution for
(\ref{SPDE-manifold}) with $r_0 = h$.
\end{definition}

In what follows, let $V \subset H$ be a finite dimensional linear
subspace and $d := \dim V$.

\begin{definition}
A family $(\mathcal{M}_t)_{t \geq 0}$ of affine subspaces
$\mathcal{M}_t \subset H$, $t \geq 0$ is called a {\rm foliation
generated by $V$} if there exists $\psi \in C^1(\mathbb{R}_+;H)$
such that
\begin{align}\label{def-foliation}
\mathcal{M}_t = \psi(t) + V, \quad t \geq 0.
\end{align}
The map $\psi$ is a {\rm parametrization} of the foliation
$(\mathcal{M}_t)_{t \geq 0}$.
\end{definition}

\begin{remark}\label{remark-para}
Note that the parametrization of a foliation $(\mathcal{M}_t)_{t
\geq 0}$ generated by $V$ is not unique. However, due to condition
(\ref{def-foliation}), for two parametrizations $\psi_1$, $\psi_2$
we have
\begin{align*}
\psi_1(t) - \psi_2(t) \in V \quad \text{for all $t \geq 0$.}
\end{align*}
\end{remark}

In what follows, let $(\mathcal{M}_t)_{t \geq 0}$ be a foliation
generated by $V$. For every $t \geq 0$ the set $\pi_{V^{\perp}}
\mathcal{M}_t$ consists of exactly one point. Therefore, the map
\begin{align*}
\psi : \mathbb{R}_+ \rightarrow H, \quad \psi(t) := \pi_{V^{\perp}}
\mathcal{M}_t
\end{align*}
is well-defined, and it is the unique parametrization of the
foliation $(\mathcal{M}_t)_{t \geq 0}$ such that $\psi(t) \in
V^{\perp}$ for all $t \geq 0$.

\begin{definition}
For each $t \geq 0$ we define the {\em tangent space}
\begin{align*}
T \mathcal{M}_t := \psi'(t) + V.
\end{align*}
\end{definition}

By Remark \ref{remark-para}, the definition of the tangent is
independent of the choice of the parametrization.


\begin{definition}\label{def-inv-foliation}
The foliation $(\mathcal{M}_t)_{t \geq 0}$ of submanifolds is {\rm
invariant} for (\ref{SPDE-manifold}) if for every $t_0 \in
\mathbb{R}_+$ and $h \in \mathcal{M}_{t_0}$ we have
\begin{align}\label{cond-foliation}
\mathbb{P}(r_t \in \mathcal{M}_{t_0 + t}) = 1 \quad \text{for all $t
\geq 0$}
\end{align}
where $(r_t)_{t \geq 0}$ denotes the weak solution for
(\ref{SPDE-manifold}) with $r_0 = h$.
\end{definition}

As we shall see now, an invariant foliation generated by $V$,
provided it exists, is unique.

\begin{lemma}\label{lemma-unique-foli}
Let $(\mathcal{M}_t^i)_{t \geq 0}$, $i=1,2$ be two foliations
generated by $V$ with $\mathcal{M}_0^1 \cap \mathcal{M}_0^2 \neq
\emptyset$, which are invariant for (\ref{SPDE-manifold}). Then we
have $\mathcal{M}_t^1 = \mathcal{M}_t^2$ for all $t \geq 0$.
\end{lemma}

\begin{proof}
Choose $h_0 \in \mathcal{M}_0^1 \cap \mathcal{M}_0^2$ and let
$(r_t)_{t \geq 0}$ be the weak solution for (\ref{SPDE-manifold})
with $r_0 = h_0$. Then we have
\begin{align*}
\pi_{V^{\perp}} \mathcal{M}_t^1 = \pi_{V^{\perp}} r_t =
\pi_{V^{\perp}} \mathcal{M}_t^2, \quad t \geq 0
\end{align*}
which completes the proof.
\end{proof}

\begin{proposition}\label{prop-transform}
Suppose the foliation $(\mathcal{M}_t)_{t \geq 0}$ of submanifolds
is invariant for (\ref{SPDE-manifold}) and let $\ell \in
L(H;\mathbb{R}^d)$ be a continuous linear operator with $\ell(V) =
\mathbb{R}^d$. Then, for every $t_0 \in \mathbb{R}_+$ and $h \in
\mathcal{M}_{t_0}$ we have almost surely
\begin{align}\label{decomposition}
r_t = \pi_{V^{\perp}} \mathcal{M}_{t_0 + t} + \ell^{-1}(\ell(r_t) -
\ell \pi_{V^{\perp}} \mathcal{M}_{t_0 + t}), \quad t \geq 0
\end{align}
where $(r_t)_{t \geq 0}$ denotes the weak solution for
(\ref{SPDE-manifold}) with $r_0 = h$, and (\ref{decomposition}) is
the decomposition of $(r_t)_{t \geq 0}$ according to $V^{\perp}
\oplus V$.
\end{proposition}

\begin{proof}
By condition (\ref{cond-foliation}) we obtain almost surely
\begin{align}\label{decomp-1}
r_t = \pi_{V^{\perp}} r_t + \pi_V r_t = \pi_{V^{\perp}}
\mathcal{M}_{t_0 + t} + \pi_V r_t, \quad t \geq 0.
\end{align}
Therefore we obtain almost surely
\begin{align}\label{decomp-2}
\pi_V r_t = r_t - \pi_{V^{\perp}} \mathcal{M}_{t_0 + t} = \ell^{-1}(
\ell(r_t) - \ell(\pi_{V^{\perp}} \mathcal{M}_{t_0 + t})), \quad t
\geq 0.
\end{align}
Inserting (\ref{decomp-2}) into (\ref{decomp-1}), we arrive at
(\ref{decomposition}).
\end{proof}

\begin{remark}
If the foliation $(\mathcal{M}_t)_{t \geq 0}$ is invariant for
(\ref{SPDE-manifold}), then for every continuous linear operator
$\ell \in L(H;\mathbb{R}^d)$ with $\ell(V) = \mathbb{R}^d$ the
decomposition (\ref{decomposition}) provides a realization of the
solution $(r_t)_{t \geq 0}$ by means of the finite dimensional
process $\ell(r)$.
\end{remark}




We shall now approach our main result of this section, Theorem
\ref{thm-foliation} below, which provides consistency conditions for
invariance of the foliation $(\mathcal{M}_t)_{t \geq 0}$.

\begin{lemma}\label{lemma-conv-linear}
There exist $\zeta_1,\ldots,\zeta_d \in \mathcal{D}(A^*)$ and an
isomorphism $\phi : \mathbb{R}^d \rightarrow V$ such that
\begin{align}\label{conv-linear}
\phi(\langle \zeta,h \rangle) = h \quad \text{for all $h \in V$,}
\end{align}
where we use the notation $\langle \zeta,h \rangle := (\langle
\zeta_1,h \rangle,\ldots,\langle \zeta_d,h \rangle) \in
\mathbb{R}^d$.
\end{lemma}

\begin{proof}
By the Gram-Schmidt method, there exists an orthonormal basis $\{
e_1,\ldots,e_d \}$ of $V$. Since $\mathcal{D}(A^*)$ is dense in $H$,
there exist $\zeta_1,\ldots,\zeta_d \in \mathcal{D}(A^*)$ with $\|
\zeta_i - e_i \| < 2^{-d}$ for $i = 1,\ldots,d$. Hence, we obtain
\begin{align*}
|\langle \zeta_i,e_j \rangle| \leq |\langle e_i,e_j \rangle| +
|\langle \zeta_i-e_i,e_j \rangle| < 2^{-d}
\end{align*}
for all $i,j = 1,\ldots,d$ with $i \neq j$ and
\begin{align*}
|\langle \zeta_i,e_i \rangle| \geq |\langle e_i,e_i \rangle| -
|\langle \zeta_i-e_i,e_i \rangle| > 1 - 2^{-d}
\end{align*}
for all $i = 1,\ldots,d$. Thus, we have
\begin{align*}
\sum_{j=1 \atop j \neq i}^d |\langle \zeta_i,e_j \rangle| < (d-1)
2^{-d} < (2^d - 1) 2^{-d} = 1 - 2^d < |\langle \zeta_i,e_i \rangle|
\end{align*}
for all $i = 1,\ldots,d$, and hence, due to the Theorem of
Gerschgorin, the $(d \times d)$-matrix
\begin{align*}
B := \left(
\begin{array}{ccc}
\langle \zeta_1,e_1 \rangle & \cdots & \langle \zeta_1,e_d \rangle
\\ \vdots & & \vdots
\\ \langle \zeta_d,e_1 \rangle & \cdots & \langle \zeta_d,e_d \rangle
\end{array}
\right)
\end{align*}
is invertible. Let $\psi : V \rightarrow \mathbb{R}^d$ be the
isomorphism
\begin{align*}
\psi(h) := (\langle e_1,h \rangle,\ldots,\langle e_d,h \rangle),
\quad h \in V.
\end{align*}
Then, the isomorphism $B \circ \psi : V \rightarrow \mathbb{R}^d$ has the
representation
\begin{align*}
(B \circ \psi)(h) = (\langle \zeta_1,h \rangle,\ldots,\langle \zeta_d,h
\rangle), \quad h \in V.
\end{align*}
Defining the isomorphism $\phi := (B \circ \psi)^{-1} : \mathbb{R}^d
\rightarrow V$ completes the proof.
\end{proof}

Now, let $\psi \in C^1(\mathbb{R}_+;H)$ be a parametrization of
$(\mathcal{M}_t)_{t \geq 0}$ and let $\phi : \mathbb{R}^d
\rightarrow V$ be an isomorphism as in Lemma
\ref{lemma-conv-linear}. We define $\tilde{\alpha},\tilde{\sigma} :
\mathbb{R}_+ \times \mathbb{R}^d \rightarrow \mathbb{R}^d$ as
\begin{align*}
\tilde{\alpha}(t,z) &:= \langle A^* \zeta, \psi(t) + \phi(z) \rangle
+ \langle \zeta, \alpha(\psi(t) + \phi(z)) - \psi'(t) \rangle,
\\ \tilde{\sigma}(t,z) &:= \langle
\zeta, \sigma(\psi(t) + \phi(z)) \rangle.
\end{align*}
By Assumption \ref{ass-SPDE} we have $\tilde{\alpha},\tilde{\sigma}
\in C^{0,1}(\mathbb{R}_+ \times \mathbb{R}^d;\mathbb{R}^d)$ and
there exists a constant $K > 0$ such that
\begin{align*}
\| \tilde{\alpha}(t,z_1) - \tilde{\alpha}(t,z_2) \|_{\mathbb{R}^d}
&\leq K \| z_1 - z_2 \|_{\mathbb{R}^d}
\\ \| \tilde{\sigma}(t,z_1) - \tilde{\sigma}(t,z_2) \|_{\mathbb{R}^d} &\leq K \| z_1 - z_2 \|_{\mathbb{R}^d}
\end{align*}
for all $t \in \mathbb{R}_+$ and all $z_1,z_2 \in \mathbb{R}^d$.
Thus, for each $t_0 \in \mathbb{R}_+$ and each $z_0 \in
\mathbb{R}^d$ there exists a unique strong solution for
\begin{align}\label{SDE-state}
\left\{
\begin{array}{rcl}
dZ_t & = & \tilde{\alpha}(t_0+t,Z_t)dt +
\tilde{\sigma}(t_0+t,Z_t)dW_t
\medskip
\\ Z_0 & = & z_0.
\end{array}
\right.
\end{align}
We define the vector field $\nu : \mathcal{D}(A) \rightarrow H$ as
\begin{align*}
\nu(h) := Ah + \alpha(h), \quad h \in \mathcal{D}(A).
\end{align*}

Here is our main result concerning invariance of the foliation
$(\mathcal{M}_t)_{t \geq 0}$ for the SPDE (\ref{SPDE-manifold}).

\begin{theorem}\label{thm-foliation}
The foliation $(\mathcal{M}_t)_{t \geq 0}$ is an invariant foliation
for (\ref{SPDE-manifold}) if and only if for all $t \geq 0$ we have
\begin{align}\label{domain-pre}
\mathcal{M}_t &\subset \mathcal{D}(A),
\\ \label{nu-pre} \nu ( h ) &\in T \mathcal{M}_t, \quad h \in
\mathcal{M}_t
\\ \label{sigma-pre} \sigma ( h ) &\in V, \quad h \in
\mathcal{M}_t.
\end{align}
If the previous conditions are satisfied, the map
\begin{align}\label{map-A}
\mathbb{R}_+ \rightarrow H, \quad t \mapsto A ( \pi_{V^{\perp}}
\mathcal{M}_t )
\end{align}
is continuous, and for every $t_0 \in \mathbb{R}_+$ and $h \in
\mathcal{M}_{t_0}$ the weak solution for (\ref{SPDE-manifold}) with
$r_0 = h$ is also a strong solution.
\end{theorem}

\begin{proof}
"$\Rightarrow$": Let $t_0 \in \mathbb{R}_+$ and $h_0 \in V$ be
arbitrary. Then we have $h := \psi(t_0) + h_0 \in
\mathcal{M}_{t_0}$. Let $(r_t)_{t \geq 0}$ be the weak solution for
(\ref{SPDE-manifold}) with $r_0 = h$ and set $z_0 := \langle
\zeta,h_0 \rangle$. Since $\zeta_1,\ldots,\zeta_d \in
\mathcal{D}(A^*)$ and $(\mathcal{M}_t)_{t \geq 0}$ is an invariant
foliation for (\ref{SPDE-manifold}), we obtain, by using
(\ref{conv-linear}),
\begin{align*}
\langle \zeta, r_t - \psi(t_0 + t) \rangle &= \langle \zeta, h -
\psi(t_0) \rangle + \int_0^t (\langle A^* \zeta,r_s \rangle +
\langle \zeta,\alpha(r_s) - \psi'(t_0 + s) \rangle) ds
\\ &\quad + \int_0^t \langle
\zeta,\sigma(r_{s}) \rangle dW_s
\\ &= \langle \zeta, h_0 \rangle + \int_0^t
\tilde{\alpha}(t_0+s,\langle \zeta, r_s - \psi(t_0 + s) \rangle) ds
\\ &\quad + \int_0^t \tilde{\sigma}(t_0+s,\langle \zeta, r_{s} - \psi(t_0 +
s) \rangle) dW_s.
\end{align*}
This identity shows that almost surely
\begin{align*}
Z_t = \langle \zeta,r_t - \psi(t_0+t) \rangle, \quad t \geq 0
\end{align*}
where $(Z_t)_{t \geq 0}$ denotes the strong solution for
(\ref{SDE-state}) with $Z_0 = z_0$. By (\ref{conv-linear}), we have
almost surely
\begin{align*}
\phi(Z_t) = r_t - \psi(t_0 + t), \quad t \geq 0.
\end{align*}
Let $\xi \in \mathcal{D}(A^*)$ be arbitrary. We obtain, by It\^o's
formula and applying the linear functional $\langle \xi,\cdot
\rangle$ afterwards,
\begin{equation}\label{combine-1}
\begin{aligned}
\langle \xi,r_t - \psi(t_0 + t) \rangle = \langle \xi,\phi(Z_t)
\rangle &= \langle \xi,h_0 \rangle + \int_0^t \langle
\xi,\phi(\tilde{\alpha}(t_0+s,Z_s)) \rangle ds
\\ &\quad + \int_0^t \langle \xi,
\phi(\tilde{\sigma}(t_0+s,Z_{s})) \rangle dW_s.
\end{aligned}
\end{equation}
Since $(r_t)_{t \geq 0}$ is a weak solution for
(\ref{SPDE-manifold}) with $r_0 = h$, we have
\begin{equation}\label{combine-2}
\begin{aligned}
\langle \xi, r_t - \psi(t_0+t) \rangle &= \langle \xi, h -
\psi(t_0) \rangle + \int_0^t (\langle A^* \xi,r_s \rangle + \langle
\xi, \alpha(r_s) - \psi'(t_0+s) \rangle) ds
\\ &\quad + \int_0^t \langle
\xi,\sigma(r_{s}) \rangle dW_s.
\end{aligned}
\end{equation}
Combining (\ref{combine-1}) and (\ref{combine-2}) we get
\begin{equation}\label{combined}
\begin{aligned}
0 &= \int_0^t (\langle A^* \xi,r_s \rangle + \langle \xi,
\alpha(r_s) - \psi'(t_0+s) - \phi(\tilde{\alpha}(t_0+s,Z_s))
\rangle) ds
\\ &\quad + \int_0^t \langle
\xi,\sigma(r_{s}) - \phi(\tilde{\sigma}(t_0+s,Z_{s})) \rangle dW_s.
\end{aligned}
\end{equation}
Therefore, all integrands in (\ref{combined}) vanish and, since $\xi
\in \mathcal{D}(A^*)$ was arbitrary, setting $s=0$ yields $\psi(t_0)
+ h_0 \in \mathcal{D}(A)$, proving (\ref{domain-pre}) and the
identities
\begin{align}\label{A-cont}
\nu(\psi(t_0) + h_0) &= \psi'(t_0) + \phi(\tilde{\alpha}(t_0,z_0))
\in T \mathcal{M}_{t_0},
\\ \sigma(\psi(t_0) + h_0) &= \phi(\tilde{\sigma}(t_0,z_0)) \in
V
\end{align}
which show (\ref{nu-pre}), (\ref{sigma-pre}). Furthermore, identity
(\ref{A-cont}) proves the continuity of the map defined in
(\ref{map-A}).

"$\Leftarrow$": Let $t_0 \in \mathbb{R}_+$ and $h \in
\mathcal{M}_{t_0}$ be arbitrary. There exists a unique $z_0 \in
\mathbb{R}^d$ such that $h = \psi(t_0) + \phi(z_0)$. Let $(Z_t)_{t
\geq 0}$ be the strong solution for (\ref{SDE-state}) with $Z_0 =
z_0$. It\^o's formula yields, by using
(\ref{domain-pre})--(\ref{sigma-pre}) and (\ref{conv-linear}),
\begin{align*}
&\psi(t_0 + t) + \phi(Z_t) = \psi(t_0) + \phi(z_0) + \int_0^t (
\psi'(t_0 + s) + \phi(\tilde{\alpha}(t_0+s,Z_s)) )ds
\\ &\quad + \int_0^t
\phi(\tilde{\sigma}(t_0+s,Z_{s}))dW_s
\\ &= h + \int_0^t (\psi'(t_0+s) + \phi(\langle \zeta, \nu(\psi(t_0+s) + \phi(Z_s)) - \psi'(t_0+s) \rangle))ds
\\ &\quad + \int_0^t \phi(\langle \zeta, \sigma(\psi(t_0+s) + \phi(Z_{s})) \rangle)dW_s
\\ &= h + \int_0^t (A(\psi(t_0+s) + \phi(Z_s)) + \alpha(\psi(t_0+s) +
\phi(Z_s)))ds
\\ &\quad + \int_0^t \sigma(\psi(t_0+s) +
\phi(Z_{s}))dW_s, \quad t \geq 0.
\end{align*}
By the uniqueness of solutions for (\ref{SPDE-manifold}) we obtain
almost surely
\begin{align*}
r_t = \psi(t_0 + t) + \phi(Z_t) \in \mathcal{M}_{t_0 + t}, \quad t
\geq 0
\end{align*}
where $(r_t)_{t \geq 0}$ denotes the weak solution for
(\ref{SPDE-manifold}) with $r_0 = h$, whence $(\mathcal{M}_t)_{t
\geq 0}$ is an invariant foliation, and we get that $(r_t)_{t \geq
0}$ is also a strong solution.
\end{proof}

\begin{remark}
Note that (\ref{domain-pre})--(\ref{sigma-pre}) are consistency
conditions on the tangent spaces (for related results see, e.g.,
\cite{Filipovic-inv}). Since the foliation $(\mathcal{M}_t)_{t \geq
0}$ consists of affine manifolds, we do not need a Stratonovich
correction term for the drift.
\end{remark}

Now, we express the consistency conditions from Theorem
\ref{thm-foliation} by means of a coordinate system. Let $\psi \in
C^1(\mathbb{R}_+;H)$ be a parametrization of $(\mathcal{M}_t)_{t
\geq 0}$ and let $\{ \lambda_1,\ldots,\lambda_d \}$ be a basis of
$V$.

\begin{corollary}\label{cor-foli-1}
The following statements are equivalent:
\begin{enumerate}
\item $(\mathcal{M}_t)_{t \geq 0}$ is an invariant foliation for
(\ref{SPDE-manifold}).

\item We have
\begin{align}\label{domain-psi-SPDE}
\psi(\mathbb{R}_+) &\subset \mathcal{D}(A),
\\ \label{domain-lambda-SPDE} \lambda_1,\ldots,\lambda_d &\in
\mathcal{D}(A)
\end{align}
and there exist $\mu,\gamma : \mathbb{R}_+ \times \mathbb{R}^d
\rightarrow \mathbb{R}^d$ such that
\begin{align}\label{nu-tangent-foli-1}
\nu \bigg( \psi(t) + \sum_{i=1}^d y_i \lambda_i \bigg) &= \psi'(t) +
\sum_{i=1}^d \mu_i(t,y) \lambda_i, \quad (t,y) \in \mathbb{R}_+
\times \mathbb{R}^d
\\ \label{sigma-tangent-foli-1} \sigma \bigg( \psi(t) + \sum_{i=1}^d y_i \lambda_i \bigg) &= \sum_{i=1}^d \gamma_i(t,y) \lambda_i, \quad
(t,y) \in \mathbb{R}_+ \times \mathbb{R}^d.
\end{align}
\end{enumerate}
If the previous conditions are satisfied, $\mu$ and $\gamma$ are
uniquely determined, we have $\mu,\gamma \in C^{0,1}(\mathbb{R}_+
\times \mathbb{R}^d;\mathbb{R}^d)$, there exists a constant $K > 0$
such that
\begin{align}\label{Lip-mu}
\| \mu(t,y_1) - \mu(t,y_2) \|_{\mathbb{R}^d} &\leq K \| y_1 - y_2
\|_{\mathbb{R}^d}
\\ \label{Lip-gamma} \| \gamma(t,y_1) - \gamma(t,y_2) \|_{\mathbb{R}^d} &\leq K \| y_1 - y_2 \|_{\mathbb{R}^d}
\end{align}
for all $t \in \mathbb{R}_+$ and $y_1,y_2 \in \mathbb{R}^d$, and for
every $t_0 \in \mathbb{R}_+$ and $h \in \mathcal{M}_{t_0}$ the weak
solution for (\ref{SPDE-manifold}) with $r_0 = h$ is also a strong
solution.
\end{corollary}

\begin{proof}
The asserted equivalence follows from Theorem \ref{thm-foliation}.
By the linear independence of $\lambda_1,\ldots,\lambda_d$, the
mappings $\mu$ and $\gamma$ are uniquely determined. Denoting by
$\pi : \mathbb{R}^d \rightarrow V$ the isomorphism $\pi(y) :=
\sum_{i=1}^d y_i \lambda_i$, we can express them as
\begin{align*}
\mu(t,y) &= \pi^{-1} \bigg( \nu \bigg( \psi(t) + \sum_{i=1}^d y_i
\lambda_i \bigg) - \psi'(t) \bigg),
\\ \gamma(t,y) &= \pi^{-1} \bigg( \sigma \bigg( \psi(t) + \sum_{i=1}^d y_i
\lambda_i \bigg) \bigg).
\end{align*}
Since the map defined in (\ref{map-A}) is continuous by Theorem
\ref{thm-foliation}, we have $\mu,\gamma \in C^{0,1}(\mathbb{R}_+
\times \mathbb{R}^d;\mathbb{R}^d)$ and (\ref{Lip-mu}),
(\ref{Lip-gamma}) by virtue of Assumption \ref{ass-SPDE}.
\end{proof}

Suppose the foliation $(\mathcal{M}_t)_{t \geq 0}$ is invariant for
(\ref{SPDE-manifold}). We shall now identify the underlying
coordinate process $Y$. Let $t_0 \in \mathbb{R}_+$ and $h \in
\mathcal{M}_{t_0}$ be arbitrary. There exists a unique $y \in
\mathbb{R}^d$ such that $h = \psi(t_0) + \sum_{i=1}^d y_i
\lambda_i$. Taking into account (\ref{Lip-mu}), (\ref{Lip-gamma}),
we let $(Y_t)_{t \geq 0}$ be the strong solution for
\begin{align}\label{coord-proc}
\left\{
\begin{array}{rcl}
dY_t & = & \mu(t_0+t,Y_t)dt + \gamma(t_0+t,Y_{t})dW_t
\medskip
\\ Y_0 & = & y,
\end{array}
\right.
\end{align}
where $\mu,\gamma : \mathbb{R}_+ \times \mathbb{R}^d \rightarrow
\mathbb{R}^d$ are given by (\ref{nu-tangent-foli-1}),
(\ref{sigma-tangent-foli-1}). By It\^o's formula, the process
\begin{align}\label{strong-sol-real}
r_t = \psi(t_0 + t) + \sum_{i=1}^d Y_t^i \lambda_i, \quad t \geq 0
\end{align}
is the strong solution for (\ref{SPDE-manifold}) with $r_0 = h$.

\begin{remark}\label{remark-economic}
If we think of interest rate models, the state process $Y$ has no
direct economic interpretation. Proposition \ref{prop-transform}
shows that for any continuous linear operator $\ell \in L(H;\mathbb{R}^d)$ 
with $\ell(V) = \mathbb{R}^d$ we can choose $\ell(r)$
as state process. We may think of $\ell_i(h) = \frac{1}{x_i}
\int_0^{x_i} h(\eta) d\eta$ (benchmark yields) or $\ell_i(h) =
h(x_i)$ (benchmark forward rates). We refer to \cite[Sec. 7]{Bj_La},
\cite[Prop. 5.1]{Bj_Go}, \cite[Thm. 3.3]{Bj_Sv}, \cite[Prop.
2]{Kwon_2003}, \cite[Sec. 5]{Duffie-Kan} for related results.
\end{remark}

\section{Affine realizations for general stochastic partial differential equations}\label{sec-real-general}

The results of the previous section lead to the following definition
of an affine realization.

\begin{definition}
Let $V \subset H$ be a finite dimensional linear subspace. The SPDE
(\ref{SPDE-manifold}) has an {\em affine realization generated by
$V$} if for each $h_0 \in \mathcal{D}(A)$ there exists a foliation
$(\mathcal{M}_t^{h_0})_{t \geq 0}$ generated by $V$ with $h_0 \in
\mathcal{M}_0^{h_0}$, which is invariant for (\ref{SPDE-manifold}).
\end{definition}

We call $d := \dim V$ the {\em dimension} of the affine realization.

\begin{lemma}\label{lemma-ddv}
Let $d \in \mathbb{N}$ and $\lambda_1,\ldots,\lambda_d \in H$ be
linearly independent. Suppose the SPDE (\ref{SPDE-manifold}) has a
$d$-dimensional affine realization generated by $V = \langle
\lambda_1,\ldots,\lambda_d \rangle$. Then, there exist
$\Phi_1,\ldots,\Phi_d \in C^1(H;\mathbb{R})$ such that
\begin{align}\label{sigma-ddv}
\sigma(h) = \sum_{i=1}^d \Phi_i(h) \lambda_i, \quad h \in H.
\end{align}
\end{lemma}

\begin{proof}
Relation (\ref{sigma-pre}) from Theorem \ref{thm-foliation} yields
$\sigma(h) \in V$ for all $h \in \mathcal{D}(A)$. Since
$\mathcal{D}(A)$ is dense in $H$ and $V$ is closed, we obtain
$\sigma(h) \in V$ for all $h \in H$. Hence, there exist
$\Phi_1,\ldots,\Phi_d : H \rightarrow \mathbb{R}$ such that
(\ref{sigma-ddv}) is satisfied. Since $\sigma \in C^1(H)$, we have
$\Phi_1,\ldots,\Phi_d \in C^1(H;\mathbb{R})$.
\end{proof}

Suppose the SPDE (\ref{SPDE-manifold}) has an affine realization
generated by a finite dimensional subspace $V \subset H$. Then, for
each $h_0 \in \mathcal{D}(A)$ the foliation
$(\mathcal{M}_t^{h_0})_{t \geq 0}$ is uniquely determined by Lemma
\ref{lemma-unique-foli}. We define the {\em singular set} $\Sigma$
as
\begin{equation}\label{def-singular-set}
\begin{aligned}
\Sigma &= \{ h_0 \in \mathcal{D}(A) : \mathcal{M}_0^{h_0} =
\mathcal{M}_t^{h_0} \text{ for all $t \geq 0$} \}
\\ &= \{ h_0 \in \mathcal{D}(A) : h_0 + V \text{ is an invariant manifold}
\}.
\end{aligned}
\end{equation}
From a geometric point of view, the singular set consists of all starting points $h_0 \in \mathcal{D}(A)$, for which
the corresponding foliation $(\mathcal{M}_t^{h_0})_{t \geq 0}$ only consists of a single
leaf, that is, the solution process even stays on the $d$-dimensional affine space $h_0 + V$. For $h_0 \in \Sigma$ the mappings $\mu,\gamma : \mathbb{R}^d \rightarrow \mathbb{R}^d$ in (\ref{coord-proc}) do not depend on the time $t$, whence the coordinate process $Y$ is time-homogeneous, and the parametrization $\psi$ in the affine realization (\ref{strong-sol-real}) may be chosen as $\psi \equiv h_0$.

A consequence of the definition of the singular set in (\ref{def-singular-set}) is the identity
\begin{align}\label{Sigma-plus-V}
\Sigma + V = \Sigma.
\end{align}
In particular, $\Sigma$ is an invariant set for
(\ref{SPDE-manifold}).

\begin{proposition}\label{prop-singular-set}
Suppose the SPDE (\ref{SPDE-manifold}) has an affine realization
generated by $V$. Then, the singular set $\Sigma$ is given by
\begin{align}\label{singular-set-general}
\Sigma = \{ h_0 \in \mathcal{D} ( A ) : \nu(h_0) \in V \},
\end{align}
for each $h_0 \in \mathcal{D}(A)$ and $t \in \mathbb{R}_+$ we have either $\mathcal{M}_t^{h_0} \cap \Sigma = \emptyset$ or $\mathcal{M}_t^{h_0} \subset \Sigma$, and for every $t_0 \in \mathbb{R}_+$ with $\mathcal{M}_{t_0}^{h_0} \subset \Sigma$ we have $\mathcal{M}_t^{h_0} = \mathcal{M}_{t_0}^{h_0}$ for all $t \geq t_0$.
\end{proposition}

\begin{proof}
Let $h_0 \in \mathcal{D}(A)$ be arbitrary. By condition
(\ref{nu-pre}) of Theorem \ref{thm-foliation} we have $\nu(h_0) \in
V$ if and only if $\nu(h) \in V$ for all $h \in h_0 + V$, which
means that $h_0 + V$ is an invariant manifold, proving
(\ref{singular-set-general}). Taking into account
(\ref{Sigma-plus-V}), we obtain the remaining statements.
\end{proof}

\begin{remark}\label{remark-singular-set}
Suppose the SPDE (\ref{SPDE-manifold}) has an affine realization
generated by $V$. For any $h_0 \in \mathcal{D}(A)$ we define the deterministic stopping time
\begin{align*}
t_0 = \inf \{ t \geq 0 : \psi(t) \in \Sigma \} \in [0,\infty]
\end{align*}
which, by Remark \ref{remark-para} and (\ref{Sigma-plus-V}), does not depend on the choice of the parametrization $\psi$ of $(\mathcal{M}_t^{h_0})_{t \geq 0}$. By Proposition \ref{prop-singular-set}, the strong solution $(r_t)_{t \geq 0}$ for (\ref{SPDE-manifold}) with $r_0 = h_0$ has the dichotomic behaviour
\begin{align}\label{dich-sets-1}
\mathbb{P}(r_t \notin \Sigma) &= 1, \quad t \in [0,t_0)
\\ \label{dich-sets-2} \mathbb{P}(r_t \in \Sigma) &= 1, \quad t \in [t_0,\infty)
\end{align}
i.e., up to time $t_0$, the solution proceeds
outside the singular set $\Sigma$, afterwards it stays in $\Sigma$, and
therefore even on an affine manifold. In particular, if $t_0 = 0$ we
have $\mathbb{P}(r_t \in \Sigma) = 1$ for all $t \geq 0$, and if
$t_0 = \infty$ we have $\mathbb{P}(r_t \notin \Sigma) = 1$ for all
$t \geq 0.$
\end{remark}


For our later investigations on the existence of affine
realizations, quasi-exponential functions (cf. \cite[Sec.
5]{Bj_Sv}), which we shall now introduce in this general context,
will play an important role. Inductively, we define the domains
\begin{align*}
\mathcal{D}(A^n) := \{ h \in \mathcal{D}(A^{n-1}) : A^{n-1} h \in
\mathcal{D}(A) \}, \quad n \geq 2
\end{align*}
as well as the intersection
\begin{align*}
\mathcal{D}(A^{\infty}) := \bigcap_{n \in \mathbb{N}}
\mathcal{D}(A^n).
\end{align*}

\begin{definition}
An element $h \in \mathcal{D}(A^{\infty})$ is called {\em
quasi-exponential} if
\begin{align}\label{QE-def}
\dim \langle A^n h : n \in \mathbb{N}_0 \rangle < \infty.
\end{align}
\end{definition}

\begin{lemma}\label{lemma-QE}
Let $h \in H$ be arbitrary. The following statements are
equivalent.
\begin{enumerate}
\item $h$ is quasi-exponential.

\item There exists $d \in \mathbb{N}$ such that $h \in \mathcal{D}(A^d)$ and $A^d h \in \langle h,Ah,\ldots,A^{d-1}h \rangle$.



\item There exists a finite dimensional subspace $V \subset \mathcal{D}(A)$ with $h \in V$ such that
\begin{align}\label{QE-cond}
A v \in V \quad \text{for all $v \in V$.}
\end{align}
\end{enumerate}
\end{lemma}

\begin{proof}
(1) $\Rightarrow$ (2): This is clear for $h=0$, and for $h \neq 0$ there
exists, by (\ref{QE-def}), a minimal integer $d \in \mathbb{N}$ such that $h,A
h,\ldots,A^{d-1} h$ are linearly independent. Consequently, we have $A^d h \in \langle h,Ah,\ldots,A^{d-1}h \rangle$.


(2) $\Rightarrow$ (3): The finite dimensional subspace $V = \langle
h,Ah,\ldots,A^{d-1}h \rangle$ has the desired properties.

(3) $\Rightarrow$ (1): Using (\ref{QE-cond}), by induction, for each
$n \in \mathbb{N}$ we have $h \in \mathcal{D}(A^n)$ and $A^n h \in
V$, which yields $h \in \mathcal{D}(A^{\infty})$ and (\ref{QE-def}),
whence $h$ is quasi-exponential.
\end{proof}

In the subsequent sections, $H$ will be a function space and $A = \frac{d}{dx}$ the differential operator.
Then, the domain $\mathcal{D}(A^{\infty})$ consists of all $C^{\infty}$-functions such that any derivative belongs to the function space $H$.
As Lemma \ref{lemma-QE} shows, a function $h \in \mathcal{D}(A^{\infty})$ is quasi-exponential if it satisfies
a linear ordinary differential equation of $d$-th order
\begin{align}\label{ODE}
A^d h = \sum_{n=0}^{d-1} c_n A^n h
\end{align}
for some $d \in \mathbb{N}$. In particular, any exponential function has this property, which explains the term \textit{quasi-exponential}. Note that (\ref{ODE}) implies that the finite dimensional subspace $V = \langle h,Ah,\ldots,A^{d-1}h \rangle \subset \mathcal{D}(A)$ is invariant under the operator $A$, i.e., we have $A V \subset V$.

Quasi-exponential functions will play a decisive role for the characterization of term structure models with an affine realization, see the subsequent Sections \ref{sec-vol-general}--\ref{sec-vol-const}.

\section{The space of forward curves}\label{sec-space}

In this section, we define the space of forward curves, on which we
will study the HJMM equation (\ref{HJMM}) in the forthcoming
sections. These spaces have been introduced in \cite[Sec.
5]{fillnm}.

We fix an arbitrary constant $\beta > 0$. Let $H_{\beta}$ be the
space of all absolutely continuous functions $h : \mathbb{R}_+
\rightarrow \mathbb{R}$ such that
\begin{align}\label{def-norm}
\| h \|_{\beta} := \bigg( |h(0)|^2 + \int_{\mathbb{R}_+} |h'(x)|^2
e^{\beta x} dx \bigg)^{\frac{1}{2}} < \infty.
\end{align}
Let $(S_t)_{t \geq 0}$ be the shift semigroup on $H_{\beta}$ defined
by $S_t h := h(t + \cdot)$ for $t \in \mathbb{R}_+$.

Since forward curves should flatten for large time to maturity $x$,
the choice of $H_{\beta}$ is reasonable from an economic point of
view.

\begin{theorem}\label{thm-group}
Let $\beta > 0$ be arbitrary.
\begin{enumerate}
\item The space $(H_{\beta},\| \cdot \|_{\beta})$ is a separable
Hilbert space.

\item For each $x \in \mathbb{R}_+$, the point
evaluation $h \mapsto h(x) : H_{\beta} \rightarrow \mathbb{R}$ is a
continuous linear functional.

\item $(S_t)_{t \geq 0}$ is a $C_0$-semigroup on $H_{\beta}$ with infinitesimal generator $\frac{d}{dx} :
\mathcal{D}(\frac{d}{dx}) \subset H_{\beta} \rightarrow H_{\beta}$,
$\frac{d}{dx}h = h'$, and domain
\begin{align*}
\mathcal{D} \bigg( \frac{d}{dx} \bigg) = \{ h \in H_{\beta} : h' \in
H_{\beta} \}.
\end{align*}

\item Each $h \in H_{\beta}$ is continuous, bounded and
the limit $h(\infty) := \lim_{x \rightarrow \infty} h(x)$ exists.

\item $H_{\beta}^0 := \{ h \in H_{\beta} : h(\infty) = 0 \}$ is a closed subspace of $H_{\beta}$.

\item There exists a universal constant
$C > 0$, only depending on $\beta$, such that for all $h \in
H_{\beta}$ we have the estimate
\begin{align}\label{estimate-c2}
\| h \|_{L^{\infty}(\mathbb{R}_+)} &\leq C \| h \|_{\beta}.
\end{align}

\item For each $\beta' > \beta$, we have $H_{\beta'} \subset
H_{\beta}$ and the relation
\begin{align}\label{est-embedding}
\| h \|_{\beta} \leq \| h \|_{\beta'}, \quad h \in H_{\beta'}.
\end{align}

\end{enumerate}
\end{theorem}

\begin{proof}
Note that $H_{\beta}$ is the space $H_w$ from \cite[Sec.
5.1]{fillnm} with weight function $w(x) = e^{\beta x}$, $x \in
\mathbb{R}_+$. Hence, the first six statements follow from
\cite[Thm. 5.1.1, Cor. 5.1.1]{fillnm}. For each $\beta' > \beta$,
the observation
\begin{align*}
\int_{\mathbb{R}_+} |h'(x)|^2 e^{\beta x} dx \leq
\int_{\mathbb{R}_+} |h'(x)|^2 e^{\beta' x} dx, \quad h \in
H_{\beta'}
\end{align*}
shows $H_{\beta'} \subset H_{\beta}$ and (\ref{est-embedding}).
\end{proof}

\begin{lemma}\label{lemma-product-pre}
The following statements are valid.
\begin{enumerate}
\item For all $h,g \in H_{\beta}$ we have $hg \in H_{\beta}$ and the
multiplication map $m : H_{\beta} \times H_{\beta} \rightarrow
H_{\beta}$ defined as $m(h,g) := hg$ is a continuous, bilinear
operator.

\item For all $h,g \in \mathcal{D}(\frac{d}{dx})$ we have $hg \in
\mathcal{D}(\frac{d}{dx})$.
\end{enumerate}
\end{lemma}

\begin{proof}
The function $hg$ is absolutely continuous, because $h$ and $g$ are
absolutely continuous and bounded, see Theorem \ref{thm-group}. By
estimate (\ref{estimate-c2}) we obtain
\begin{align*}
\| hg \|_{\beta}^2 &= |h(0)|^2 |g(0)|^2 + \int_{\mathbb{R}_+} | h(x)
g'(x) + g(x) h'(x) |^2 e^{\beta x} dx
\\ &\leq \| h \|_{L^{\infty}(\mathbb{R}_+)}^2 \| g \|_{L^{\infty}(\mathbb{R}_+)}^2 + 2
\| h \|_{L^{\infty}(\mathbb{R}_+)}^2 \int_{\mathbb{R}_+} |g'(x)|^2
e^{\beta x} dx
\\ & \quad + 2 \| g \|_{L^{\infty}(\mathbb{R}_+)}^2
\int_{\mathbb{R}_+} |h'(x)|^2 e^{\beta x} dx
\\ &\leq C^4 \| h \|_{\beta}^2 \| g \|_{\beta}^2 + 2 C^2
\| h \|_{\beta}^2 \| g \|_{\beta}^2 + 2 C^2 \| g \|_{\beta}^2 \| h
\|_{\beta}^2 < \infty.
\end{align*}
Hence, we have $hg \in H_{\beta}$ and the
estimate
\begin{align*}
\| m(h,g) \|_{\beta} \leq \sqrt{C^4 + 4 C^2} \| h \|_{\beta} \| g
\|_{\beta},
\end{align*}
proving that $m$ is a continuous, bilinear operator.

If $h,g \in \mathcal{D}(\frac{d}{dx})$, we have $hg \in
C^1(\mathbb{R}_+)$ with $(hg)' = h'g + hg'$, whence $hg \in
\mathcal{D}(\frac{d}{dx})$ by the first statement.
\end{proof}

For $\lambda \in H_{\beta}$ we define $\Lambda := \mathcal{I}
\lambda := \int_0^{\bullet} \lambda(\eta) d\eta$, which belongs to
$C(\mathbb{R}_+)$, the space of all continuous functions from
$\mathbb{R}_+$ to $\mathbb{R}$.

\begin{lemma}\label{lemma-cont-int}
Let $0 < \beta < \beta'$ be arbitrary real numbers. For each
$\lambda \in H_{\beta'}^0$ we have $\Lambda \in H_{\beta}$ and the
map $\mathcal{I} : H_{\beta'}^0 \rightarrow H_{\beta}$ is a
continuous linear operator.
\end{lemma}

\begin{proof}
Let $\lambda \in H_{\beta'}^0$ be arbitrary. Then $\mathcal{I}
\lambda$ is absolutely continuous. Since $\mathcal{I}\lambda(0) =
0$, using the Cauchy Schwarz inequality, we obtain
\begin{align*}
\| \mathcal{I} \lambda \|_{\beta}^2 &= \int_{\mathbb{R}_+}
\lambda(x)^2 e^{\beta x} dx = \int_{\mathbb{R}_+} \left(
\int_x^{\infty} \lambda'(y) e^{\frac{1}{2} \beta' y} e^{-\frac{1}{2}
\beta' y} dy \right)^2 e^{\beta x} dx
\\ &\leq \int_{\mathbb{R}_+} \left( \int_x^{\infty} \lambda'(y)^2 e^{\beta' y} dy \right) \left( \int_x^{\infty} e^{-\beta' y} dy
\right) e^{\beta x} dx
\\ &\leq \| \lambda \|_{\beta'}^2 \int_{\mathbb{R}_+} \frac{1}{\beta'} e^{-(\beta' - \beta)x} dx
= \frac{1}{\beta'(\beta' - \beta)} \| \lambda \|_{\beta'}^2,
\end{align*}
proving the assertion.
\end{proof}


\section{Invariant foliations for the HJMM
equation}\label{sec-HJMM-fol}

We shall now investigate invariant foliations for the HJMM equation
(\ref{HJMM}) by working on the space of forward curves from the
previous section.

Let $0 < \beta < \beta'$ be arbitrary real numbers and let $\sigma :
H_{\beta} \rightarrow H_{\beta}$ be given.

\begin{assumption}\label{ass-HJMM}
We assume that $\sigma \in C^1(H_{\beta})$ with $\sigma(H_{\beta})
\subset H_{\beta'}^0$ and that there exist $L,M > 0$ such that
\begin{align*}
\| \sigma(h_1) - \sigma(h_2) \|_{\beta} &\leq L \| h_1 - h_2
\|_{\beta} \quad \text{for all $h_1,h_2 \in H_{\beta}$,}
\\ \| \sigma(h) \|_{\beta} &\leq M \quad \text{for all $h \in H_{\beta}$.}
\end{align*}
\end{assumption}

Using the notation of the previous section, the HJM drift term
(\ref{HJM-drift}) is given by
\begin{align*}
\alpha_{\rm HJM} = m(\sigma,\mathcal{I}\sigma).
\end{align*}
Recall that this choice of the drift ensures that the implied bond market (\ref{bond-market}) will be free of arbitrage opportunities.

According to \cite[Cor. 5.1.2]{fillnm} we have $\alpha_{\rm
HJM}(H_{\beta}) \subset H_{\beta}^0$ and there exists a constant
$K > 0$ such that
\begin{align*}
\| \alpha_{\rm HJM}(h_1) - \alpha_{\rm HJM}(h_2) \|_{\beta} &\leq K
\| h_1 - h_2 \|_{\beta} \quad \text{for all $h_1,h_2 \in
H_{\beta}$.}
\end{align*}
Hence, for each $h_0 \in H_{\beta}$ there exists a unique weak
solution for (\ref{HJMM}) with $r_0 = h_0$, see \cite[Thm. 6.5, Thm. 7.4]{Da_Prato}. Note that (\ref{HJMM})
is a particular example of the stochastic partial differential
equation (\ref{SPDE-manifold}) on the state space $H = H_{\beta}$ with generator $A = \frac{d}{dx}$ and drift $\alpha =
\alpha_{\rm HJM}$. Moreover, Lemmas \ref{lemma-product-pre},
\ref{lemma-cont-int} yield $\alpha_{\rm HJM} \in C^1(H_{\beta})$,
whence all required conditions from Assumption \ref{ass-SPDE} are
fulfilled.

Now let $(\mathcal{M}_t)_{t \geq 0}$ be a foliation generated by a
finite dimensional subspace $V \subset H_{\beta}$. We set $d := \dim
V$. In order to investigate invariance of $(\mathcal{M}_t)_{t \geq
0}$ for the HJMM equation (\ref{HJMM}), we directly switch to a
coordinate system. Let $\psi \in C^1(\mathbb{R}_+;H)$ be a
parametrization of $(\mathcal{M}_t)_{t \geq 0}$ and let $\{
\lambda_1,\ldots,\lambda_d \}$ be a basis of $V$. Then, the set $\{
\Lambda_1,\ldots,\Lambda_d \}$ is linearly independent in
$C(\mathbb{R}_+)$.

\begin{remark}\label{remark-lin-ind}
Let $E_1 \subset \{ 1,\ldots,d \}$ be an index set. We set
\begin{align*}
E_2 := \{ (i,j) \in E_1 \times E_1 : i < j \}.
\end{align*}
Then, there are subsets $D_1 \subset E_1$ and $D_2 \subset E_2$ such
that
\begin{align}\label{def-B}
B = \{ \Lambda_1,\ldots,\Lambda_d \} \cup \{ \Lambda_i^2 : i \in D_1
\} \cup \{ \Lambda_i \Lambda_j : (i,j) \in D_2 \}
\end{align}
is a basis of the vector space
\begin{align}\label{def-V}
W = \langle \Lambda_1,\ldots,\Lambda_d \rangle + \langle \Lambda_i^2
: i \in E_1 \rangle + \langle \Lambda_i \Lambda_j : (i,j) \in E_2
\rangle.
\end{align}
For each $m \in E_1 \setminus D_1$ there exist unique
$(c_{i}^{m})_{i = 1,\ldots,d} \subset \mathbb{R}$, $(d_{i}^{m})_{i
\in D_1} \subset \mathbb{R}$ and $(d_{ij}^{m})_{(i,j) \in D_2}
\subset \mathbb{R}$ such that
\begin{align}\label{lin-ind-square}
\Lambda_m^2 = \sum_{i=1}^d c_i^{m} \Lambda_i + \sum_{i \in D_1}
d_{i}^{m} \Lambda_i^2 + \sum_{(i,j) \in D_2} d_{ij}^{m} \Lambda_i
\Lambda_j,
\end{align}
and for each $(m,n) \in E_2 \setminus D_2$ there exist unique
$(c_{i}^{mn})_{i = 1,\ldots,d} \subset \mathbb{R}$, $(d_{i}^{mn})_{i
\in D_1} \subset \mathbb{R}$ and $(d_{ij}^{mn})_{(i,j) \in D_2}
\subset \mathbb{R}$ such that
\begin{align}\label{lin-ind-mixed}
\Lambda_m \Lambda_n = \sum_{i=1}^d c_i^{mn} \Lambda_i + \sum_{i \in
D_1} d_{i}^{mn} \Lambda_i^2 + \sum_{(i,j) \in D_2} d_{ij}^{mn}
\Lambda_i \Lambda_j.
\end{align}
\end{remark}

\begin{theorem}\label{thm-crucial}
The foliation $(\mathcal{M}_t)_{t \geq 0}$ is invariant for the HJMM
equation (\ref{HJMM}) if and only if we have
(\ref{domain-psi-SPDE}), (\ref{domain-lambda-SPDE}), there exist
$\mu,\gamma \in C^{0,1}(\mathbb{R}_+ \times
\mathbb{R}^d;\mathbb{R}^d)$ such that
\begin{align}\label{nu-at-zero}
\nu ( \psi(t) ) = \psi'(t) + \sum_{i=1}^d \mu_i(t,0) \lambda_i,
\quad t \in \mathbb{R}_+
\end{align}
and (\ref{sigma-tangent-foli-1}), there are $(a_i^k)_{i =
1,\ldots,d}^{k = 1,\ldots,d} \subset \mathbb{R}$, $(b_i^k)_{i \in
D_1}^{k = 1,\ldots,d} \subset \mathbb{R}$ and $(b_{ij}^k)_{(i,j) \in
D_2}^{k=1,\ldots,d} \subset \mathbb{R}$ such that for all $(t,y) \in
\mathbb{R}_+ \times \mathbb{R}^d$ we have
\begin{align}\label{const-Riccati-a}
&-\frac{\partial}{\partial y_k} \mu_i(t,y) + \frac{1}{2} \sum_{m \in
E_1 \setminus D_1} c_i^{m} \frac{\partial}{\partial
y_k}(\gamma_m(t,y)^2)
\\ \notag &+ \sum_{(m,n) \in E_2 \setminus D_2}
c_i^{mn} \frac{\partial}{\partial y_k}(\gamma_m(t,y) \gamma_n(t,y))
= a_i^k, \quad i=1,\ldots,d, \, k=1,\ldots,d
\\ \label{const-Riccati-b1} &\frac{1}{2} \frac{\partial}{\partial y_k}
(\gamma_i(t,y)^2) + \frac{1}{2} \sum_{m \in E_1 \setminus D_1}
d_i^{m} \frac{\partial}{\partial y_k}(\gamma_m(t,y)^2)
\\ \notag &+ \sum_{(m,n)
\in E_2 \setminus D_2} d_i^{mn} \frac{\partial}{\partial
y_k}(\gamma_m(t,y) \gamma_n(t,y)) = b_{i}^k, \quad i \in D_1, \,
k=1,\ldots,d
\\ \label{const-Riccati-b2} & \frac{\partial}{\partial y_k} (\gamma_i(t,y)\gamma_j(t,y)) +
\frac{1}{2} \sum_{m \in E_1 \setminus D_1} d_{ij}^{m}
\frac{\partial}{\partial y_k}(\gamma_m(t,y)^2)
\\ \notag &+ \sum_{(m,n) \in E_2 \setminus D_2}
d_{ij}^{mn} \frac{\partial}{\partial y_k}(\gamma_m(t,y)
\gamma_n(t,y)) = b_{ij}^k, \quad (i,j) \in D_2, \, k=1,\ldots,d
\end{align}
where $E_1 \subset \{ 1,\ldots,d \}$ is chosen such that $E_1
\supset \{ i = 1,\ldots,d : \gamma_i \not \equiv 0 \}$ and the
further quantities are chosen as in Remark \ref{remark-lin-ind}, and
we have the Riccati equations
\begin{align}\label{Riccati}
\frac{d}{dx} \Lambda_k + \sum_{i=1}^d a_i^k \Lambda_i + \sum_{i \in
D_1} b_i^k \Lambda_i^2 + \sum_{(i,j) \in D_2} b_{ij}^k \Lambda_i
\Lambda_j = \lambda_k(0), \quad k=1,\ldots,d.
\end{align}
\end{theorem}

\begin{proof}
"$\Rightarrow$" Suppose $(\mathcal{M}_t)_{t \geq 0}$ is an invariant
foliation for (\ref{HJMM}). According to Corollary \ref{cor-foli-1}
we have (\ref{domain-psi-SPDE})--(\ref{sigma-tangent-foli-1}).
Relation (\ref{nu-at-zero}) follows by setting $y=0$ in
(\ref{nu-tangent-foli-1}). Inserting (\ref{sigma-tangent-foli-1})
into (\ref{nu-tangent-foli-1}) we get, by taking into account the
HJM drift condition (\ref{HJM-drift}),
\begin{align*}
\frac{d}{dx} \psi(t) + \sum_{i=1}^d y_i \frac{d}{dx} \lambda_i +
\frac{1}{2} \frac{d}{dx} \bigg( \sum_{i=1}^d \gamma_i(t,y) \Lambda_i
\bigg)^2 = \psi'(t) + \sum_{i=1}^d \mu_i(t,y) \lambda_i
\end{align*}
for all $(t,y) \in \mathbb{R}_+ \times \mathbb{R}^d$.
Differentiating with respect to $y_k$ we obtain
\begin{align*}
&\frac{d}{dx} \lambda_k + \frac{d}{dx} \bigg( \bigg( \sum_{i=1}^d
\gamma_i(t,y) \Lambda_i \bigg) \bigg( \sum_{i=1}^d
\frac{\partial}{\partial y_k} \gamma_i(t,y) \Lambda_i \bigg) \bigg)
\\ &= \sum_{i=1}^d \frac{\partial}{\partial y_k} \mu_i(t,y) \lambda_i,
\quad k=1,\ldots,d
\end{align*}
for all $(t,y) \in \mathbb{R}_+ \times \mathbb{R}^d$. Integrating
yields
\begin{align*}
\lambda_k - \sum_{i=1}^d \frac{\partial}{\partial y_k} \mu_i(t,y)
\Lambda_i + \sum_{i,j=1}^d \gamma_i(t,y) \frac{\partial}{\partial
y_k} \gamma_j(t,y) \Lambda_i \Lambda_j = \lambda_k(0), \quad k =
1,\ldots,d
\end{align*}
for all $(t,y) \in \mathbb{R}_+ \times \mathbb{R}^d$. Noting that
$E_1 \supset \{ i = 1,\ldots,d : \gamma_i \not \equiv 0 \}$, we can
express this equation as
\begin{align*}
&\lambda_k - \sum_{i=1}^d \frac{\partial}{\partial y_k} \mu_i(t,y)
\Lambda_i + \frac{1}{2} \sum_{i \in E_1} \frac{\partial}{\partial
y_k} (\gamma_i(t,y)^2) \Lambda_i^2
\\ &+ \sum_{(i,j) \in E_2}
\frac{\partial}{\partial y_k} (\gamma_i(t,y)\gamma_j(t,y)) \Lambda_i
\Lambda_j = \lambda_k(0), \quad k = 1,\ldots,d
\end{align*}
for all $(t,y) \in \mathbb{R}_+ \times \mathbb{R}^d$. Introducing
the functions $f_i : \mathbb{R}_+ \times \mathbb{R}^d \rightarrow
\mathbb{R}^d$, $i = 1,\ldots,d$ and $g_i : \mathbb{R}_+ \times
\mathbb{R}^d \rightarrow \mathbb{R}^d$, $i \in D_1$ as well as
$g_{ij} : \mathbb{R}_+ \times \mathbb{R}^d \rightarrow
\mathbb{R}^d$, $(i,j) \in D_2$ by
\begin{align*}
f_i(t,y) &:= -\mu_i(t,y) + \frac{1}{2} \sum_{m \in E_1 \setminus
D_1} c_i^{m} \gamma_m(t,y)^2 + \sum_{(m,n) \in E_2 \setminus D_2}
c_i^{mn} \gamma_m(t,y) \gamma_n(t,y),
\\ g_i(t,y) &:= \frac{1}{2} \gamma_i(t,y)^2 + \frac{1}{2} \sum_{m \in E_1 \setminus D_1} d_i^{m}
\gamma_m(t,y)^2 + \sum_{(m,n) \in E_2 \setminus D_2} d_i^{mn}
\gamma_m(t,y) \gamma_n(t,y),
\\ g_{ij}(t,y) &:= \gamma_i(t,y)\gamma_j(t,y) + \frac{1}{2}
\sum_{m \in E_1 \setminus D_1} d_{ij}^{m} \gamma_m(t,y)^2
\\ &\quad + \sum_{(m,n) \in E_2 \setminus D_2} d_{ij}^{mn} \gamma_m(t,y)
\gamma_n(t,y),
\end{align*}
we obtain, by taking into account (\ref{lin-ind-square}) and
(\ref{lin-ind-mixed}),
\begin{align*}
&\sum_{i=1}^d \frac{\partial}{\partial y_k} f_i(t,y) \Lambda_i +
\sum_{i \in D_1} \frac{\partial}{\partial y_k} g_i(t,y) \Lambda_i^2
+ \sum_{(i,j) \in D_2} \frac{\partial}{\partial y_k} g_{ij}(t,y)
\Lambda_i \Lambda_j
\\ &= -(\lambda_k - \lambda_k(0)), \quad k=1,\ldots,d
\end{align*}
for all $(t,y) \in \mathbb{R}_+ \times \mathbb{R}^d$. Since $B$
defined in (\ref{def-B}) is a basis of the vector space $W$ in
(\ref{def-V}), we deduce (\ref{const-Riccati-a}),
(\ref{const-Riccati-b1}), (\ref{const-Riccati-b2}) and the Riccati
equations (\ref{Riccati}).

"$\Leftarrow$": Relations (\ref{HJM-drift}),
(\ref{sigma-tangent-foli-1}), (\ref{lin-ind-square}),
(\ref{lin-ind-mixed})  yield
\begin{equation}\label{alpha}
\begin{aligned}
&\alpha_{\rm HJM}\bigg( \psi(t) + \sum_{i=1}^d y_i \lambda_i \bigg)
= \frac{1}{2} \frac{d}{dx} \bigg( \sum_{i,j=1}^d \gamma_i(t,y)
\gamma_j(t,y) \Lambda_i \Lambda_j \bigg)
\\ &= \frac{1}{2} \frac{d}{dx} \bigg( \sum_{i \in E_1} \gamma_i(t,y)^2
\Lambda_i^2 + 2 \sum_{(i,j) \in E_2} \gamma_i(t,y) \gamma_j(t,y)
\Lambda_i \Lambda_j \bigg)
\\ &= \frac{d}{dx} \bigg( \sum_{i=1}^d (\mu_i(t,y) + f_i(t,y)) \Lambda_i + \sum_{i \in D_1}
g_i(t,y) \Lambda_i^2 + \sum_{(i,j) \in D_2} g_{ij}(t,y) \Lambda_i
\Lambda_j \bigg)
\end{aligned}
\end{equation}
for all $(t,y) \in \mathbb{R}_+ \times \mathbb{R}^d$. In particular,
by setting $y = 0$, we have
\begin{equation}\label{alpha-zero}
\begin{aligned}
&\alpha_{\rm HJM}( \psi(t) )
\\ &= \frac{d}{dx} \bigg( \sum_{i=1}^d
(\mu_i(t,0) + f_i(t,0)) \Lambda_i + \sum_{i \in D_1} g_i(t,0)
\Lambda_i^2 + \sum_{(i,j) \in D_2} g_{ij}(t,0) \Lambda_i \Lambda_j
\bigg)
\end{aligned}
\end{equation}
for all $t \in \mathbb{R}_+$. Relations (\ref{alpha}),
(\ref{const-Riccati-a}), (\ref{const-Riccati-b1}),
(\ref{const-Riccati-b2}), (\ref{alpha-zero}) and the Riccati
equations (\ref{Riccati}) give us
\begin{align*}
&\alpha_{\rm HJM}\bigg( \psi(t) + \sum_{i=1}^d y_i \lambda_i \bigg)
= \frac{d}{dx} \bigg( \sum_{i=1}^d \bigg( \mu_i(t,y) + f_i(t,0) +
\sum_{k=1}^d a_i^k y_k \bigg) \Lambda_i
\\ &\quad + \sum_{i \in D_1} \bigg(
g_i(t,0) + \sum_{k=1}^d b_i^k y_k \bigg) \Lambda_i^2 + \sum_{(i,j)
\in D_2} \bigg( g_{ij}(t,0) + \sum_{k=1}^d b_{ij}^k y_k \bigg)
\Lambda_i \Lambda_j \bigg)
\\ &= \frac{d}{dx} \bigg( \sum_{i=1}^d ( \mu_i(t,0) + f_i(t,0) ) \Lambda_i +
\sum_{i \in D_1} g_i(t,0) \Lambda_i^2 + \sum_{(i,j) \in D_2}
g_{ij}(t,0) \Lambda_i \Lambda_j \bigg)
\\ &\quad + \frac{d}{dx} \sum_{i=1}^d (\mu_i(t,y) - \mu_i(t,0)) \Lambda_i
\\ &\quad + \frac{d}{dx} \sum_{k=1}^d y_k \bigg( \sum_{i=1}^d a_i^k \Lambda_i
+ \sum_{i \in D_1} b_i^k \Lambda_i^2 + \sum_{(i,j) \in D_2} b_{ij}^k
\Lambda_i \Lambda_j \bigg)
\\ &= \alpha_{\rm HJM}(\psi(t)) + \sum_{i=1}^d (\mu_i(t,y) -
\mu_i(t,0)) \lambda_i - \sum_{k=1}^d y_k \frac{d}{dx} \lambda_k
\end{align*}
for all $(t,y) \in \mathbb{R}_+ \times \mathbb{R}^d$. We conclude,
by furthermore incorporating (\ref{nu-at-zero}),
\begin{align*}
&\nu \bigg( \psi(t) + \sum_{i=1}^d y_i \lambda_i \bigg) =
\frac{d}{dx} \psi(t) + \sum_{i=1}^d y_i \frac{d}{dx} \lambda_i +
\alpha_{\rm HJM}\bigg( \psi(t) + \sum_{i=1}^d y_i \lambda_i \bigg)
\\ &= \nu(\psi(t)) + \sum_{i=1}^d (\mu_i(t,y) -
\mu_i(t,0)) \lambda_i = \psi'(t) + \sum_{i=1}^d \mu_i(t,y) \lambda_i
\end{align*}
for all $(t,y) \in \mathbb{R}_+ \times \mathbb{R}^d$, showing
(\ref{nu-tangent-foli-1}). According to Corollary \ref{cor-foli-1},
the foliation $(\mathcal{M}_t)_{t \geq 0}$ is an invariant for
(\ref{HJMM}).
\end{proof}

Note that in particular the system (\ref{Riccati}) of Riccati
equations is useful in order to gain knowledge about the existence
of an affine realization. We will exemplify Theorem
\ref{thm-crucial} in the subsequent sections in order to
characterize volatility structures for which the HJMM equation
(\ref{HJMM}) admits an affine realization. We will start with general
volatilities in Section \ref{sec-vol-general}, and will obtain results for particular volatility
structures as corollaries in Sections \ref{sec-cdv}--\ref{sec-short-rate}.

\section{Affine realizations for the HJMM equation with general volatility}\label{sec-vol-general}

In this section, we assume that the volatility $\sigma$ in the
HJMM equation (\ref{HJMM}) is of the form
\begin{align}\label{sigma-ddv-type}
\sigma(h) = \sum_{i=1}^p \Phi_i(h) \lambda_i, \quad h \in H_{\beta}
\end{align}
where $p \in \mathbb{N}$ denotes a positive integer,
$\Phi_1,\ldots,\Phi_p : H_{\beta} \rightarrow \mathbb{R}$ are
functionals and $\lambda_1,\ldots,\lambda_p \in H_{\beta'}^0$ are
linearly independent. We assume that $\Phi_i \in
C^2(H_{\beta};\mathbb{R})$ for $i=1,\ldots,p$ and that there exist
$L,M > 0$ such that for all $i=1,\ldots,p$ we have
\begin{align*}
|\Phi_i(h_1) - \Phi_i(h_2)| &\leq L \| h_1 - h_2 \|_{\beta} \quad
\text{for all $h_1,h_2 \in H_{\beta}$,}
\\ |\Phi_i(h)| &\leq M \quad \text{for all $h \in H_{\beta}$.}
\end{align*}
Then, Assumption \ref{ass-HJMM} is fulfilled.

Note that, in view of Lemma \ref{lemma-ddv}, this is the most
general volatility, which we can have for the HJMM equation
(\ref{HJMM}) with an affine realization. The corresponding HJM drift
term (\ref{HJM-drift}) is given by
\begin{align}\label{HJM-drift-ddv}
\alpha_{\rm HJM}(h) = \bigg( \sum_{i=1}^p \Phi_i(h) \lambda_i \bigg)
\bigg( \sum_{i=1}^p \Phi_i(h) \Lambda_i \bigg), \quad h \in
H_{\beta}.
\end{align}

\begin{proposition}\label{prop-general-vol}
Suppose there exist $h_1,\ldots,h_p \in H_{\beta}$ such that
$\sigma(h_1),\ldots,\sigma(h_p)$ are linearly independent, and $h_0
\in \mathcal{D}(\frac{d}{dx})$ such that one of the following
conditions is satisfied:
\begin{itemize}
\item We have
\begin{align}\label{prod-Phi-diff}
D(\Phi_i \Phi_j)(h_0)\lambda_k = 0, \quad i,j,k = 1,\ldots,p.
\end{align}
\item There exist $k,l \in \{ 1,\ldots,p
\}$ such that the functions
\begin{align}\label{prod-Phi-lin-ind}
h \mapsto D^2(\Phi_i \Phi_j) (h) (\lambda_k,\lambda_l), \quad h \in
h_0 + \langle \lambda_1,\ldots,\lambda_p \rangle
\end{align}
are linearly independent for $1 \leq i \leq j \leq p$.
\end{itemize}
If the HJMM equation (\ref{HJMM}) has an affine realization, then
$\lambda_1,\ldots,\lambda_p$ are quasi-exponential.
\end{proposition}

\begin{proof}
Let $V \subset H_{\beta}$ be a finite dimensional subspace
generating the affine realization and set $d := \dim V$. Lemma
\ref{lemma-ddv} yields that $\sigma(h) \in V$ for all $h \in
H_{\beta}$. Since $\sigma(h_1),\ldots,\sigma(h_p)$ are linearly
independent, we obtain $\lambda_1,\ldots,\lambda_p \in V$, because 
relation (\ref{sigma-ddv-type}) yields that
\begin{align*}
\langle \lambda_1,\ldots,\lambda_p \rangle = \langle \sigma(h_1),\ldots,\sigma(h_p) \rangle.
\end{align*}
Choose $\lambda_{p+1},\ldots,\lambda_d \in V$ such that $\{
\lambda_1,\ldots,\lambda_d \}$ is a basis of $V$. Let $h_0 \in
\mathcal{D}(\frac{d}{dx})$ be such that one of the conditions above
is satisfied. Now we apply Theorem \ref{thm-crucial} to the
invariant foliation $(\mathcal{M}_t^{h_0})_{t \geq 0}$. In view of (\ref{sigma-tangent-foli-1}) and (\ref{sigma-ddv-type}),
the function $\gamma = \gamma(0,\cdot) : \mathbb{R}^d \rightarrow \mathbb{R}^d$ is given by
\begin{align*}
\gamma_i(y) &= \Phi_i \bigg( h_0 + \sum_{i=1}^d y_i \lambda_i
\bigg), \quad i=1,\ldots,p
\\ \gamma_i(y) &= 0, \quad i=p+1,\ldots,d.
\end{align*}
In particular, we have $\gamma \in C^{2}(\mathbb{R}^d)$ and we can choose $E_1 = \{ 1,\ldots,p \}$.

If (\ref{prod-Phi-diff}) is satisfied, then
(\ref{const-Riccati-b1}), (\ref{const-Riccati-b2}) give us $b_i^k =
0$ for all $i \in D_1$, $k=1,\ldots,p$ and $b_{ij}^k = 0$ for all
$(i,j) \in D_2$, $k=1,\ldots,p$. Consequently, the Riccati equations
(\ref{Riccati}) show that $\lambda_1,\ldots,\lambda_p$ are
quasi-exponential.

If there exist $k,l \in \{ 1,\ldots,p \}$ such that the functions
(\ref{prod-Phi-lin-ind}) are linearly independent for $1 \leq i \leq
j \leq p$, then we claim that $D_1 = D_2 = \emptyset$, which, in
view of the Riccati equations (\ref{Riccati}), implies that
$\lambda_1,\ldots,\lambda_p$ are quasi-exponential. Suppose, on the
contrary, that $D_1 \neq \emptyset$ or $D_2 \neq \emptyset$.

If $D_1 \neq \emptyset$, choose $i \in D_1$ and differentiate
(\ref{const-Riccati-b1}) with respect to $y_l$, which yields
\begin{align*}
&\frac{1}{2} D^2 \Phi_i^2(h)(\lambda_k,\lambda_l) + \frac{1}{2}
\sum_{m \in E_1 \setminus D_1} d_i^m D^2
\Phi_m^2(h)(\lambda_k,\lambda_l)
\\ &+ \sum_{(m,n) \in E_2 \setminus D_2} d_i^{mn} D^2(\Phi_m
\Phi_n)(h)(\lambda_k,\lambda_l) = 0
\end{align*}
for all $h \in h_0 + \langle \lambda_1,\ldots,\lambda_p \rangle$.
This contradicts the linear independence of (\ref{prod-Phi-lin-ind})
for $1 \leq i \leq j \leq p$.

Analogously, if $D_2 \neq \emptyset$, choosing $(i,j) \in D_2$ and
differentiating (\ref{const-Riccati-b2}) with respect to $y_l$
yields a contradiction to the linear independence of
(\ref{prod-Phi-lin-ind}) for $1 \leq i \leq j \leq p$.
\end{proof}

\begin{proposition}\label{prop-general-suff}
If $\lambda_1,\ldots,\lambda_p$ are quasi-exponential, then the HJMM
equation (\ref{HJMM}) has an affine realization.
\end{proposition}

\begin{proof}
Since $\lambda_1,\ldots,\lambda_p$ are quasi-exponential, the linear space
\begin{align*}
W := \Big\langle \bigg( \frac{d}{dx} \bigg)^n \lambda_1 : n \in \mathbb{N}_0 \Big\rangle + \ldots + \Big\langle \bigg( \frac{d}{dx} \bigg)^n \lambda_p : n \in \mathbb{N}_0 \Big\rangle \subset \mathcal{D} \bigg( \frac{d}{dx} \bigg)
\end{align*}
is finite dimensional and we have
\begin{align}\label{diff-W}
\frac{d}{dx} w \in W \quad \text{for all $w \in W$.}
\end{align}
Since $S_t H_{\beta'}^0 \subset H_{\beta'}^0$ for all $t \geq 0$, we
have $W \subset H_{\beta'}^0$. Set $q := \dim W$. There exist
$\lambda_{p+1},\ldots,\lambda_q \in W$ such that $\{
\lambda_1,\ldots,\lambda_q \}$ is a basis of $W$. We define the
subspace
\begin{align*}
V := W + \langle \lambda_i \Lambda_j : i,j = 1,\ldots,q \rangle.
\end{align*}
Note that $V \subset \mathcal{D}( \frac{d}{dx} )$ by Lemmas
\ref{lemma-product-pre}, \ref{lemma-cont-int}. Set $d := \dim V$ and
choose $\lambda_{q+1},\ldots,\lambda_d \in V$ such that $\{
\lambda_1,\ldots,\lambda_d \}$ is a basis of $V$. Relation
(\ref{diff-W}) implies
\begin{align}\label{span-Lambda}
\frac{d}{dx} \Lambda_i \in \langle 1,\Lambda_1,\ldots,\Lambda_q
\rangle, \quad i=1,\ldots,q.
\end{align}
By (\ref{span-Lambda}) and (\ref{diff-W}) we have
\begin{align*}
\frac{d}{dx}(\lambda_i \Lambda_j) = \lambda_i \frac{d}{dx} \Lambda_j
+ \Lambda_j \frac{d}{dx} \lambda_i \in V, \quad i,j=1,\ldots,q
\end{align*}
whence we have
\begin{align}\label{diff-V}
\frac{d}{dx} v \in V \quad \text{for all $v \in V$.}
\end{align}
Let $h_0 \in \mathcal{D}(\frac{d}{dx})$ be arbitrary. We define the
map $\psi \in C^1(\mathbb{R}_+;H_{\beta})$,
\begin{align}\label{psi-general}
\psi(t) &:= S_t h_0
\end{align}
the map $\gamma \in C^{0,1}(\mathbb{R}_+ \times
\mathbb{R}^d;\mathbb{R}^d)$,
\begin{align}\label{gamma-general} \gamma_i(t,y) &:=
\begin{cases}
\Phi_i ( \psi(t) + \sum_{j=1}^d y_j \lambda_j ), & i = 1,\ldots,p
\\ 0, & i = p+1,\ldots,d.
\end{cases}
\end{align}
and $\mu \in C^{0,1}(\mathbb{R}_+ \times \mathbb{R}^d;\mathbb{R}^d)$
as
\begin{align}\label{mu-general}
\mu_i(t,y) &:=
\begin{cases}
\gamma_k(t,y) \gamma_l(t,y) + \sum_{j=1}^d a_{ji} y_j & \text{if $i
\in \{ q+1,\ldots,d \}$ and $\lambda_i = \lambda_k \Lambda_l$}
\\ & \text{for
some $k,l \in \{ 1,\ldots,p \}$,}
\\ \sum_{j=1}^d a_{ji} y_j, & \text{otherwise,}
\end{cases}
\end{align}
where, due to (\ref{diff-V}), the $(a_{ij})_{i,j=1,\ldots,d} \subset
\mathbb{R}$ are chosen such that
\begin{align*}
\frac{d}{dx} \lambda_i = \sum_{j=1}^d a_{ij} \lambda_j, \quad
i=1,\ldots,d.
\end{align*}
Then, conditions
(\ref{domain-psi-SPDE})--(\ref{sigma-tangent-foli-1}) are fulfilled,
and therefore, by Corollary \ref{cor-foli-1}, the foliation
$(\mathcal{M}_t^{h_0})_{t \geq 0}$ generated by $V = \langle
\lambda_1,\ldots,\lambda_d \rangle$ with parametrization $\psi$ is
invariant for the HJMM equation (\ref{HJMM}).
\end{proof}

\begin{remark}\label{remark-geom-0}
Note that the proof of Proposition \ref{prop-general-suff}
simultaneously provides the construction of the affine realization.
For $h_0 \in \mathcal{D}(\frac{d}{dx})$ the invariant foliation
$(\mathcal{M}_t^{h_0})_{t \geq 0}$ is generated by $\langle
\lambda_1,\ldots,\lambda_d \rangle$ and has the parametrization
$\psi$ defined in (\ref{psi-general}). For $h \in
\mathcal{M}_{t_0}^{h_0}$ with some $t_0 \in \mathbb{R}_+$ the strong
solution $(r_t)_{t \geq 0}$ for (\ref{HJMM}) with $r_0 = h$ is given
by (\ref{strong-sol-real}), where the maps $\mu,\gamma :
\mathbb{R}_+ \times \mathbb{R}^d \rightarrow \mathbb{R}^d$ for the
state process (\ref{coord-proc}) are defined in
(\ref{gamma-general}), (\ref{mu-general}). We refer to \cite[Prop.
5.1, Prop. 6.1]{Bj_La} for similar results.
\end{remark}

\begin{remark}\label{remark-ec-state}
According to Remark \ref{remark-economic}, for any continuous linear operator $\ell \in L(H_{\beta};\mathbb{R}^d)$ with $\ell(V) = \mathbb{R}^d$ we can choose $\ell(r)$ as state process. For example, the components $\ell_i$ could be evaluations of benchmark yields or benchmark forward rates. This provides an economic interpretation of the affine realization.
\end{remark}

\begin{remark}\label{remark-geom-1}
Combining Proposition \ref{prop-singular-set} and relations
(\ref{HJM-drift-ddv}), (\ref{psi-general}), the singular set
$\Sigma$ is given by the $(d+1)$-dimensional linear space
\begin{align*}
\Sigma = \langle 1, \Lambda_1,\ldots,\Lambda_d \rangle,
\end{align*}
and, by Remark \ref{remark-singular-set}, for each $h_0 \in \mathcal{D}(\frac{d}{dx})$ we have
(\ref{dich-sets-1}), (\ref{dich-sets-2}), where $t_0$ denotes the deterministic stopping time
\begin{align*}
t_0 = \inf \{ t \geq 0 : S_t h_0 \in \Sigma \} \in [0,\infty]
\end{align*}
and where $(r_t)_{t \geq 0}$ denotes the strong solution for
(\ref{HJMM}) with $r_0 = h_0$.
\end{remark}

\begin{remark}\label{remark-supp-1}
Note that the conditions in Proposition \ref{prop-general-vol} are
singular events, because the respective conditions only have to be
satisfied for one single point $h_0$. Hence, Propositions
\ref{prop-general-vol}, \ref{prop-general-suff} yield that, apart
from degenerate examples like the CIR model, the existence of an
affine realization is essentially equivalent to the condition that
$\lambda_1,\ldots,\lambda_p$ are quasi-exponential (which means that
all quadratic terms in the system (\ref{Riccati}) of Riccati
equations disappear). This also supplements \cite[Prop. 6.4]{Bj_Sv},
which provides the sufficient implication.
\end{remark}

\section{Affine realizations for the HJMM equation with constant direction volatility}\label{sec-cdv}

In this section, we study the existence of affine realizations
for the HJMM equation (\ref{HJMM}) with constant direction
volatility, that is, we assume that the volatility $\sigma$ in the
HJMM equation (\ref{HJMM}) is of the form
\begin{align}\label{sigma-ddv-single}
\sigma(h) = \Phi(h) \lambda, \quad h \in H_{\beta}
\end{align}
where $\Phi : H_{\beta} \rightarrow \mathbb{R}$ is a functional and
$\lambda \in H_{\beta'}^0$ with $\lambda \neq 0$. We assume that
$\Phi \in C^2(H_{\beta};\mathbb{R})$ and that there exist $L,M > 0$
such that
\begin{align*}
| \Phi(h_1) - \Phi(h_2) | &\leq L \| h_1 - h_2 \|_{\beta} \quad
\text{for all $h_1,h_2 \in H_{\beta}$,}
\\ | \Phi(h) | &\leq M \quad \text{for all $h \in H_{\beta}$.}
\end{align*}
Then, Assumption \ref{ass-HJMM} is fulfilled.

\begin{corollary}\label{cor-ddv-nec}
Suppose $\Phi \not \equiv 0$. If the HJMM equation (\ref{HJMM}) has
an affine realization, then $\lambda$ is quasi-exponential or we
have
\begin{align}\label{CIR-1}
D^2 \Phi^2(h)(\lambda,\lambda) &= 0, \quad h \in \mathcal{D} \bigg(
\frac{d}{dx} \bigg) \quad \text{and}
\\ \label{CIR-2} D \Phi^2(h)\lambda &\neq 0, \quad h \in \mathcal{D} \bigg(
\frac{d}{dx} \bigg).
\end{align}
\end{corollary}

\begin{proof}
This is an immediate consequence of Proposition
\ref{prop-general-vol}.
\end{proof}

\begin{remark}
Conditions (\ref{CIR-1}), (\ref{CIR-2}) mean that at each forward
curve $h \in \mathcal{D}(\frac{d}{dx})$ the functional $\Phi^2$ is affine, but not
constant, in direction $\lambda$, which is the typical feature for
CIR type models.
\end{remark}

\begin{corollary}\label{cor-ddv-suff}
If $\lambda$ is quasi-exponential, then the HJMM equation
(\ref{HJMM}) has an affine realization.
\end{corollary}

\begin{proof}
This is a direct consequence of Proposition \ref{prop-general-suff}.
\end{proof}

\begin{remark}\label{remark-supp-2}
Suppose we have $\Phi \not \equiv 0$ and there exists $h_0 \in
\mathcal{D}(\frac{d}{dx})$ such that
\begin{align*}
D^2 \Phi^2(h_0)(\lambda,\lambda) \neq 0 \quad \text{or} \quad D
\Phi^2(h_0) \lambda = 0.
\end{align*}
Then, by Corollaries \ref{cor-ddv-nec}, \ref{cor-ddv-suff}, the HJMM
equation (\ref{HJMM}) has an affine realization if and only if
$\lambda$ is quasi-exponential. Hence, we have relaxed the
assumptions from \cite[Prop. 6.1]{Bj_Sv}, where it is assumed that
$\Phi(h) \neq 0$ for all $h \in H_{\beta}$ and $D^2
\Phi^2(h)(\lambda,\lambda) \neq 0$ for all $h \in H_{\beta}$.
\end{remark}

\section{Affine realizations for the HJMM equation with constant volatility}\label{sec-vol-const}

In this section, we study the existence of affine realizations
for the HJMM equation (\ref{HJMM}) with constant volatility, i.e.,
we have
\begin{align*}
\sigma \equiv \lambda
\end{align*}
with $\lambda \in H_{\beta'}^0$, $\lambda \neq 0$. Then, Assumption
\ref{ass-HJMM} is fulfilled.

\begin{corollary}\label{cor-det}
The HJMM equation (\ref{HJMM}) has an affine realization if and only
if $\lambda$ is quasi-exponential.
\end{corollary}

\begin{proof}
The assertion is a direct consequence of Corollaries
\ref{cor-ddv-nec}, \ref{cor-ddv-suff}, because $\sigma$ is of the
form (\ref{sigma-ddv-single}) with $\Phi \equiv 1$.
\end{proof}

\begin{remark}\label{remark-geom-2}
If $\lambda$ is quasi-exponential, we even obtain a $d$-dimensional
affine realization, where $d := \dim \langle (\frac{d}{dx})^n
\lambda : n \in \mathbb{N}_0 \rangle$. For $h_0 \in
\mathcal{D}(\frac{d}{dx})$ the invariant foliation
$(\mathcal{M}_t^{h_0})_{t \geq 0}$ is generated by $\langle
\lambda_1,\ldots,\lambda_d \rangle$ with
\begin{align*}
\lambda_i = \bigg( \frac{d}{dx} \bigg)^{i-1} \lambda, \quad
i=1,\ldots,d
\end{align*}
and has the parametrization
\begin{align*}
\psi(t) = - \frac{1}{2} \Lambda^2 + S_t \bigg( h_0 + \frac{1}{2}
\Lambda^2 \bigg), \quad t \in \mathbb{R}_+
\end{align*}
which can be shown by using Corollary \ref{cor-foli-1} (cf.
\cite[Prop. 4.1]{Bj_La}). Using Proposition \ref{prop-singular-set},
the singular set $\Sigma$ is given by the $(d+1)$-dimensional affine
space
\begin{align*}
\Sigma = - \frac{1}{2} \Lambda^2 + \langle 1,
\Lambda_1,\ldots,\Lambda_d \rangle,
\end{align*}
and, by Remark \ref{remark-singular-set}, for each $h_0 \in \mathcal{D}(\frac{d}{dx})$ we have
(\ref{dich-sets-1}), (\ref{dich-sets-2}), where $t_0$ denotes the deterministic stopping time
\begin{align*}
t_0 = \inf \Big\{ t \geq 0 : S_t \bigg( h_0 + \frac{1}{2} \Lambda^2
\bigg) \in \langle 1, \Lambda_1,\ldots,\Lambda_d \rangle \Big\} \in
[0,\infty]
\end{align*}
and where $(r_t)_{t \geq 0}$ denotes the strong solution for
(\ref{HJMM}) with $r_0 = h_0$.
\end{remark}

\begin{remark}
Not surprisingly, the statement of Corollary \ref{cor-det} coincides
with that of \cite[Prop. 5.1]{Bj_Sv}.
\end{remark}

\section{Short rate realizations for the HJMM equation}\label{sec-short-rate}

In this last section, we deal with affine realizations of dimension $d=1$. As explained in Remark \ref{remark-ec-state},
we can give an economic interpretation to the affine realization and choose (subject to slight regularity conditions) the
short rate $r(0)$ as state process. In this case, we also speak about a \textit{short rate realization}.

Let us assume that the volatility $\sigma$ is of the form
\begin{align}\label{vol-CIR-form}
\sigma(h) = \phi(h(0)), \quad h \in H_{\beta}
\end{align}
where $h \mapsto h(0) : H_{\beta} \rightarrow \mathbb{R}$ denotes the evaluation of the short rate and where $\phi : \mathbb{R} \rightarrow H_{\beta}$ is an arbitrary map. Then, the short rate process will be the solution of a one-dimensional stochastic differential equation.

Using our previous results with $d=1$ and taking into account the particular structure (\ref{vol-CIR-form}) of the volatility, we see that $\sigma$ is of one of the following three types: 

\begin{enumerate}
\item We can have
\begin{align*}
\sigma \equiv \rho
\end{align*}
with a constant $\rho \in \mathbb{R}$. This is the Ho-Lee model.

\item We can have
\begin{align*}
\sigma \equiv \rho e^{-c \bullet}
\end{align*}
with appropriate constants $\rho,c \in \mathbb{R}$. This is
the Hull-White extension of the Vasicek model.

\item We can have
\begin{align*}
\sigma(h) = \sqrt{c + d h(0)} \lambda, \quad h \in H_{\beta}
\end{align*}
with appropriate constants $c,d \in \mathbb{R}$, where $\Lambda = \int_0^{\bullet} \lambda(\eta) d\eta$ satisfies a Riccati equation of the kind
\begin{align*}
\frac{d}{dx} \Lambda + a \Lambda + b \Lambda^2 = \lambda(0).
\end{align*}
This is the Hull-White extension of the Cox-Ingersoll-Ross model.
\end{enumerate}

Thus, we have recognized the three well-known short rate models, which is completely in line with the existing literature, see,
e.g., \cite{Jeffrey, Bj_Sv, Filipovic-Teichmann-royal}.

\section{Conclusion}\label{sec-conclusion}

We have presented an alternative approach on the existence of affine
realizations for HJM interest rate models, which has the feature to be
applicable to be a wide class of models and being conceptually rather comprehensible.

Applying this approach, we have been able to provide further insights into the
structure of affine realizations. In particular, we have seen that essentially all volatility structures with an affine realization are of the form (\ref{sigma-ddv-type}) with $\lambda_1,\ldots,\lambda_p$ being quasi-exponential. All remaining volatilities with an affine realization, like the CIR model, may be considered as degenerate examples, see Remark \ref{remark-supp-1}. 

Our proofs have provided constructions of the affine realizations (see Remarks \ref{remark-geom-0}, \ref{remark-ec-state}) and we have been able to determine the singular set $\Sigma$, see Remark \ref{remark-geom-1}, where we have also exhibited the dichotomic behaviour of the forward rate process with respect to $\Sigma$. Moreover, for particular volatility
structures we have supplemented some known existence results.

\subsection*{Acknowledgement}

The author gratefully acknowledges the support from WWTF (Vienna Science and Technology Fund).

The author is also grateful to two anonymous referees for their helpful comments and suggestions.

\end{document}